\documentclass{amsart}

\usepackage[english]{babel}

\usepackage[letterpaper,top=2cm,bottom=2cm,left=3cm,right=3cm,marginparwidth=1.75cm]{geometry}

\usepackage{amsfonts} 
\usepackage{amsmath}
\usepackage{amsthm}
\usepackage{amssymb}
\usepackage{graphicx}
\usepackage{enumerate}
\usepackage[colorlinks=true, allcolors=blue]{hyperref}

\theoremstyle{plain}
\newtheorem{thm}{Theorem}[section]
\newtheorem{defi}[thm]{Definition} 
\newtheorem{prop}[thm]{Proposition}
\newtheorem{lemma}[thm]{Lemma}

\newtheorem{cor}[thm]{Corollary} 
\newtheorem{rk}[thm]{Remark}
\newtheorem{notation}[thm]{Notation}

\newcommand{\nn}{\mathbb{N}}

\newcommand{\rr}{\mathbb{R}}
\newcommand{\cc}{\mathbb{C}}

\newcommand{\mycomment}[1]{}

\newcommand{\bC}{{\mathbb{C}}}

\newcommand{\bN}{{\mathbb{N}}}

\newcommand{\bR}{{\mathbb{R}}}

  \newcommand{\A}{{\mathcal{A}}}
  \newcommand{\B}{{\mathcal{B}}}

  \newcommand{\M}{{\mathcal{M}}}
  \newcommand{\N}{{\mathcal{N}}}
\renewcommand{\O}{{\mathcal{O}}}

\renewcommand{\leq}{\leqslant}
\renewcommand{\geq}{\geqslant}

\renewcommand{\phi}{\varphi}
\newcommand{\upchi}{{\raise.35ex\hbox{\ensuremath{\chi}}}}
\newcommand{\eps}{\varepsilon}
\title{Schur tests and estimates for positive semigroups on non commutative $L_p$}
\author[C. Coine]{Cl\'ement Coine}\address{UNICAEN, CNRS, LMNO, 14000 Caen, France}
\email{clement.coine@unicaen.fr}
\author[M. Laforge]{M\'elina Laforge}\address{UNICAEN, CNRS, LMNO, 14000 Caen, France}
\email{melina.laforge@unicaen.fr}
\author[\'E. Ricard]{\'Eric Ricard}\address{UNICAEN, CNRS, LMNO, 14000 Caen, France}
\email{eric.ricard@unicaen.fr}

\begin{document}
\maketitle

\begin{abstract}
We extend the recent results of Arnold on estimates for Kreiss positive semigroups on $L_p$-spaces to the non commutative setting, $1<p<\infty$. We also provide a sharp example. All the results are achieved thanks to the use of the Schur test, including a new version for von Neumann algebras.  
\end{abstract}


\section{Introduction}

The study of the growth rate of powers of bounded operators and of $C_0$-semigroups is a central topic in operator theory. A fundamental approach to this problem is to relate the asymptotic behavior of an operator or semigroup to estimates on its resolvent, thereby connecting dynamical properties with spectral information. A cornerstone of this approach is the Hille--Yosida theorem, which characterizes the generators of bounded $C_0$-semigroups in terms of resolvent estimates. In the discrete-time setting, the Kreiss matrix theorem provides a sharp bound for the iterates of a matrix in terms of its Kreiss constant, and more generally the Kreiss condition provides a natural resolvent criterion for controlling the growth of powers of an operator.

The aim of this paper is to study the optimal growth rates for positive Kreiss bounded operator and $C_0$-semigroups on non commutative $L_p$-spaces associated with semifinite von Neumann algebras, extending recent results obtained in the commutative setting.\\

\noindent \textbf{Background and notations.} Let $X$ be a complex Banach space. For an operator $A$ on $X$ (not necessarily bounded), we let, for every $\lambda \in \rho(A)$ (the resolvent set of $A$),
$$
R(\lambda, A) :=(\lambda I - A)^{-1}.
$$
Let $T$ be a bounded linear operator on $X$ with $\sigma(T)\subset \overline{\mathbb D}$. We denote
$$K_T = \sup_{| \lambda | > 1} \, (| \lambda | -1) \| R(\lambda,T) \|$$
and we say that $T$ is Kreiss bounded if $K_T < \infty$.
For the continuous-time analogue, we say that a $C_0$-semigroup $T=(T_t)_{t\geq 0}$ with generator $A$ is Kreiss bounded if every $\lambda \in \mathbb{C}$, $\text{Re}(\lambda)>0$, is in $\rho(A)$ and $K_T < \infty$ where
$$K_T = \sup_{ \text{Re}(\lambda) > 0}  \, \text{Re}(\lambda) \| R(\lambda,A) \|.$$

The Kreiss resolvent condition is a classical spectral condition. It provides quantitative information on the growth of powers of an operator, or of a $C_0$-semigroup. If $X$ is finite dimensional, the Kreiss matrix theorem states that a matrix $M$ is power bounded (respectively, the semigroup is bounded) if and only if $M$ (resp. the semigroup) is Kreiss bounded, see e.g. \cite{LevTre}. When $X$ is infinite dimensional, this condition does not, in general, imply power boundedness. See \cite[Example 3.1]{Eisner2006} where Eisner and Zwart built a positive Kreiss bounded semigroup with exponential growth. However, on a Hilbert space or for positive semigroups and operators on commutative $L_p$-spaces, the Kreiss condition yields a more restrictive growth rate. For the Hilbert case, in \cite{Eisner2006}, the authors showed the estimate $\|T_t\| = \mathcal{O}(t)$. An improvement was obtained in \cite{Arnold2} by Arnold, who proved that $\|T_t\|=\mathcal{O}(t/ \sqrt{\log(t+1)})$ until, very recently, he showed that $\|T_t\| = \mathcal{O}(t^{1-\epsilon_K})$ where $\epsilon_K$ depends explicitly on the Kreiss constant, see \cite{Arnold3}. These improvements rely in an essential way on the Hilbert space structure, in particular on the availability of $L_2$-based arguments and Plancherel-type identities, which allow one to exploit the resolvent estimate more efficiently than in a general Banach space. Conversely, for every $\epsilon > 0$, there is a Kreiss bounded semigroup on a Hilbert space satisfying $\|T_t\| \gtrsim t^{1-\epsilon}$, see \cite[Example 4.4]{Eisner2006}.

In the discrete setting, the situation is better on general Banach spaces. Indeed, the Kreiss resolvent condition implies the linear estimate $\|T^n\| = \mathcal{O}(n)$, see \cite{LubichNev}. This order of growth cannot be improved in general without further assumptions on the Banach space; see \cite{Nev2001}. On a Hilbert space $H$, however, better upper estimates are available. In \cite{CCEL}, the authors proved that if $T$ is Kreiss bounded on $H$, then $\|T^n\|=\mathcal{O}(n/\sqrt{\log(n)})$. This bound was recently improved in \cite{Arnold3} where it was shown that for such a $T$, we have $\|T^n\| = \mathcal{O}(n^{1-\epsilon_T})$. As in the continuous case, for every $\epsilon > 0$, there is a Kreiss bounded bounded operator $T$ on a Hilbert space such that $\|T^n\| \gtrsim n^{1-\epsilon}$, see \cite{Spijker2003} or \cite[Theorem 2.3]{Bonilla2021}.

Related estimates for positive semigroups $(T_t)_{t\geq 0}$ and operators $T$ on $L_p(\Omega)$ were obtained in \cite{Arnold} where the authors showed that $\|T_t\| = \mathcal{O}(t/\log(1+t)^{1/p \wedge 1/p'})$ and $\|T^n\|=\mathcal{O}(n/\log(1+n)^{1/p \wedge 1/p'})$, until the bounds $\mathcal{O}(t^{1-\epsilon})$ and $\mathcal{O}(n^{1-\epsilon})$ were obtained in \cite{Arnold3}. These results rely on techniques developed in the study of stability of positive semigroups, notably Weis's result concerning convolution-valued functions with a positive operator kernel, see \cite{Wei98}.

In this paper, we extend these polynomial growth estimates to positive Kreiss bounded operators and $C_0$-semigroups on non commutative $L_p$-spaces associated with semifinite von Neumann algebras, for $1<p<\infty$. More precisely, we prove that
$$
\|T^n\|=\mathcal{O}(n^{1-\delta_T})
\qquad\text{and}\qquad
\|T_t\|=\mathcal{O}(t^{1-\delta_T}),
$$
respectively, for some $\delta_T\in(0,1)$. In the Hilbert space setting, our method also yields similar bounds for Kreiss bounded operators and semigroups, providing short alternative proofs of Arnold's results \cite{Arnold3}.
We provide examples showing that these bounds can not be improved for $1<p<\infty$, extending the Hilbert case. The main point here is that our arguments are rather simple based on the Schur test and characterization of Ces\`aro (or mean) bounded operators.
\\

Another classical question concerning $C_0$-semigroups is whether the spectral bound of the generator coincides with the growth bound of the semigroup. For a $C_0$-semigroup $T=(T_t)_{t\geq 0}$ with generator $A$, we define the spectral bounded of $A$ by
$$s(A)=\sup\{ {\rm Re\,} \lambda ; \;\lambda\in \sigma(A)\}$$
and the growth bound of $T$ by
$$\omega(T)=\inf\{ \omega\in \bR\, |\, \sup_{t>0} e^{-t\omega} \|T_t\|<\infty\}.$$
It is well-known that $s(A)\leq \omega(T)$ for all semigroups, see Theorem 5.1.9 in \cite[Theorem 5.1.9]{Arendt}. However, in general, there is no equality, as shown by the classical counterexample \cite{Zab}. On the other hand, it is a classical result by Weis \cite{Wei95, Wei98} that for a positive semigroup on a commutative $L_p$-space, $s(A)=\omega(T)$. In the non commutative setting, partial results were obtained in \cite{Prajapati}. We show in this paper that this equality remains valid in the non commutative setting.\\

The main tool that allows us to pass to the non commutative setting is an appropriate version of the reciprocal of the Schur test.
Usually, the Schur test gives a bound of an integral operator on a commutative $L_p$. When the kernel is positive, this is actually a characterization. In \cite{Pis94}, Pisier conceptualized it using complex interpolation of operators and obtained very nice formulations for regular operators (linear combinations of positive operators). This allowed him to
get a non commutative extension in \cite{Pis95} but only for the Schatten $p$-classes (non commutative $L_p$ based on $B(H)$ as von Neumann algebra) and for completely bounded maps. The main drawback in his approach  was the use of deep analytic factorization theorems and the structure of completely positive maps in $B(H)$. Here we adopt a much simpler approach very close to what is usually done in the commutative case where the interpolation is given by a very explicit form. As corollaries, we get answers to old problems in the field: the norm of a positive map on a non commutative $L_p$ is almost achieved on its positive cone and the complete
norm of a completely positive map coincides with its norm. We also get an estimate for operator convolutions that yields all the estimates for positive semigroups.

\bigskip

\noindent \textbf{Organization of the paper.} In Section 2, we recall the classical Schur test and establish the non commutative version that will be used later, together with several consequences for positive and completely positive maps and for operator-valued convolution. In Section 3 we study positive Kreiss bounded operators and semigroups. We first treat the discrete case, proving the equivalence between Kreiss and Cesàro boundedness and establishing the $\mathcal{O}(n^{1-\delta_T})$ estimate. Actually, all the proofs for these norme estimates follow the exact same pattern. We then turn to semigroups and obtain the corresponding $\mathcal{O}(t^{1-\delta_T})$ estimate. The final part of the section is devoted to spectral bounds and the equality between spectral and growth bounds for positive semigroups on non commutative $L_p$-spaces. Finally, Section 4 is devoted to optimal examples. We introduce the relevant weighted spaces, establish the necessary $A_p$-estimates and, by means of the Schur test, construct a positive Kreiss bounded translation semigroup whose norms grow like $t^\gamma$ for a fixed $\gamma \in (0,1)$, showing that the sublinear estimates obtained in the preceding sections are optimal.







\section{The Schur test}

 We use standard notation and refer to \cite{DPS} for non commutative $L_p$.
For the reader not familiar with this theory, one can think of 
$(M,\tau)$ as $B(H)$  with its usual (unbounded) trace as an example of semifinite von Neumann algebra. In this situation 
$(B(H),{\rm Tr})$ is finite iff $H$ is finite dimensional.

\bigskip

Let us recall the classical Schur test for a (positive) integral operator 
$T:L_p(Y)\to L_p(X)$, $1<p<\infty$, between $\sigma$-finite measure space
with measurable  kernel $K:X\times Y\to \rr^+$, i.e. for $f\in L_p(Y)^+$, $x\in X$ a.e.:
$$T(f)(x) =  \int_Y K(x,y)f(y)dy.$$

\begin{prop}[The Schur test]\label{schurtest}
  If there exist $A,B:X\times Y\rightarrow \rr^+$ measurable
  and constants $\alpha$, $\beta$ such that,
  $K\leq A^{1-\frac 1 p}B^{\frac1 p}$,
$$\int_Y  A(x,y)dy \leq \alpha \text{  a.e. in $x$}$$
and,
$$\int_X  B(x,y)dx \leq \beta \text{ a.e. in  $y$.}$$
Then, the operator $T$ can be extended to an operator on $L_p(Y)\to L_p(X)$ with norm bounded by $\alpha^{1-\frac 1p}\beta^{\frac 1p}$. 
\end{prop}

Actually this is a characterization, see \cite{Gra} Appendix I, in the following sense 

\begin{prop}
  With the notation above, assume that $\gamma>0$ and there is
  $h\in L_\infty(Y)^+$ with finite support such that $T(h)>0$ a.e.,
  then the following are equivalent
  \begin{enumerate}
  \item $T$ maps $L_p(Y)$ into $L_p(X)$ with norm at most $\gamma$.
  \item\label{eq2}  for all $C>\gamma$ there are measurable functions $u:Y\to (0,\infty)$, $v:X\to (0,\infty)$ such that
    $$T(u^{p'})\leq  Cv^{p'} \; a.e., \qquad T^*(v^{p})\leq C  u^{p}\; a.e..$$     \end{enumerate}
\end{prop}

Note that if  \eqref{eq2} is satisfied, then one can decompose
$$K(x,y)= \left(K(x,y)^{\frac 1 {p'}}\frac {u(y)}{v(x)}\right).\left(K(x,y)^{\frac 1 {p}}\frac {v(x)}{u(y)}\right):=A^{\frac 1 {p'}}.B^{\frac 1 {p}}$$
where $A$ and $B$ satisfies the hypotheses of Proposition \ref{schurtest} with $\alpha=\beta=C$.

\smallskip

We now give a non commutative version. For technical reasons, we
will stick to finite von Neumann algebras. From now on, we assume that
$T:L_p(\M)\to L_p(\N)$ is a bounded positive operator between two
finite von Neumann algebras $(\M,\tau_\M)$ and $(\N,\tau_\N)$.

We use the notation $\|T\|_+= \sup_{h\in L_p(\M)^+, \|h\|_p\leq 1} \|T(h)\|$. We recall that
the order using the Jordan decomposition yields that
\begin{equation}\label{possa}
  \|T\|_+= \sup_{h\in L_p(\M)^+, \|h\|_p\leq 1} \|T(h)\|=\sup_{h\in L_p(\M)^{s.a}, \|h\|_p\leq 1} \|T(h)\|.
  \end{equation}

\begin{lemma}\label{pertur}
  Let $\M,\,\N$ and $T$ as above with $\|T\|_+\geq 1$. Then for all $\eps\in(0,1)$, there are sequences
  $x_n\in B_{L_p(\M)^+}$ and $y_n\in B_{L_p(\N)^+}$, such that if
  $U(x)=\eps \sum_{k\geq 0} 2^{-k}
  \tau_\M(x_k^{p-1}x) y_k$, then there is $h\in L_p(\M)^+$, $\|h\|_p=1$ such that
  $$\| T+U \|_+=\|(T+U)(h)\|_p.$$
\end{lemma}
\begin{proof}
  We fix $0<\eps<1$ and we can assume $\|T\|_+=1$ by homogeneity.

  We construct by induction the sequence $(x_n)_{n\geq 0}$,
  $y_n=\frac {T_n(x_n)}{\|T_n(x_n)\|_p}$ 
  and  $$T_n(.)=T+\eps \sum_{k= 0}^{n-1} 2^{-k}   \tau_\M(x_k^{p-1}.) y_k.$$

  So let $T_0=T$ and choose $x_0\in L_p(\M)^+$ with $\|x_0\|_p=1$ and
  $\|T_0(x_0)\|_p\geq \|T_0\|_+- \eps$.  Thus $y_0$
  is well defined as $T_0(x_0)\neq 0$.

  The sequence $(x_k)$ is constructed by induction by choosing for $n\geq 0$, 
  $x_{n+1}\in L_p(\M)^+$ with $\|x_{n+1}\|_p=1$ and
  \begin{equation}\label{defrec}\|T_{n+1}(x_{n+1})\|_p\geq \|T_{n+1}\|_+-\eps 4^{-(n+1)}.\end{equation} Note that $T_n$ are positive with   $\|T_{n+1}\|_+\geq \|T_0\|_+=1$ ensuring the construction of $y_{n+1}$.

  We turn to the existence of $h$. First,  by construction, we have for $n\geq 0$
    $$ \|T_{n+1}\|_+\geq \|T_{n+1}(x_n)\|_p= \|T_n(x_n)\|_p+\eps 2^{-n}\geq \|T_n\|_+-\eps 4^{-n}+\eps 2^{-n}.$$
    And also
    $$\|T_{n+1}\|_+-\eps 4^{-(n+1)}\leq\|T_{n+1}(x_{n+1})\|_p\leq
    \|T_n\|_+ + \eps 2^{-n} \tau_\M(x_n^{p-1}x_{n+1}).$$
    Thus
    $$\|T_{n}\|_+-\eps (4^{-(n+1)}+4^{-n})+\eps 2^{-n} \leq \|T_n\|_+ + \eps 2^{-n} \tau_\M(x_n^{p-1}x_{n+1}).$$
   We get
    $$\eps 2^{-n} \tau_\M(x_n^{p-1}x_{n+1}) \geq \eps(2^{-n}- 2.4^{-n}).$$
    We have proved $\tau_\M(x_n^{p-1}x_{n+1})\geq 1-  2^{-(n-1)}$.
    
    Now we use the uniform convexity of non commutative $L_p$,
    $1<p<\infty$ in the following form. There exist $C>0$ and
    $\theta\in(0,1)$ such that if $x,y\in B_{L_p(\M)^+}$ satisfy
    $\tau(x^{p-1}y)\geq 1-\delta$ then
    $\|x-y\|_p\leq C \delta^\theta$. We deduce that
    $\|x_{n+1}-x_n\|_p\leq C  2^{-(n-1)\theta}$. The sequence
    $(x_n)$ is Cauchy, it goes to some $h\in L_p(\M)^+$ with
    $\|h\|_p=1$.  as $(T_n)$ is also converging to $T+U$ for the norm
    $B(L_p(\M),L_p(\N))$. The triangle inequality also gives that $\lim_n \|T_n\|_+=\|T+U\|_+$ and we can conclude that $$\|(T+U)(h)\|_p=\|T+U\|_+$$ taking limit in \eqref{defrec}.
\end{proof}

\begin{thm}\label{ncschur} Let $(\M, \tau_\M)$ and $(\N, \tau_\N)$ be finite von Neumann algebras,  $1<p<\infty$ and
  $T: L_p(\M)\to L_p(\N)$ a bounded positive map with $\|T\|_+=1$.

  Then for all $t>1$, there exist $u\in L_p(\M)^+$, $v\in L_p(\N)^+$ both with norm 1 and bounded inverse such that
  the maps  defined by
  $$ S_0 : \M\to L_p(\N); x\mapsto v^{-\frac 12} T(u^{\frac 12} x u^{\frac 12})v^{-\frac 12},
  \quad  S_1^* : \N\to L_{p'}(\M); y\mapsto u^{-\frac {p-1}2} T^*(v^{\frac {p-1}2} y v^{\frac {p-1}2})u^{-\frac {p-1}2}$$
  take values respectively in $\N$ and $\M$ and are positive with $\|S_0(1)\|,\|S_1^*(1)\|\leq t$ and are normal.

  If $T$ is completely positive, then $S_0$ and $S_1^*$ are also completely positive.
\end{thm}
\begin{proof}
  Let $T$ as above and $\eps\in(0,1)$, we consider
  $\tilde T= T+ \eps \tau_\M(.)1_{\N}$. This map has the advantage that for all
  $x\in L_p(\M)^+$ non zero, $\tilde T(x)$ is invertible in $\N$ and  for all
  $y\in L_{p'}(\N)^+$ non zero $\tilde T^*(y)$ is invertible in $\M$. Also
  $\|T\|_+\leq \|\tilde T\|_+\leq \|\tilde T\|_+ +\eps \tau_\M(1)\tau_\N(1)$ thus the middle quantity goes to 1 when $\eps\to 0$. 

  We apply Lemma \ref{pertur} to $\tilde T$ and $\eps$. We have $U:L_p(\M)\to L_p(\N)$
  which is positive of norm less than $2\eps$ and $h\in L_p(\M)^+$ with $\|h\|_p=1$ and
  $\|S(h)\|_p= \|S\|_+=1$ where $S=(\tilde T+U)/\|\tilde T+U\|_+$. We also have that
  $S(h)^{-1}\in \N$.

  On the real Banach space $L_p(\M)^{s.a}$, the linear
  functional $\phi: x\mapsto \tau_\N\big(S(h)^{p-1}S(x)\big)$ is
  positive of norm 1 (see \eqref{possa}) that is achieved on $h$.  By uniform convexity, we
  must have $\phi(x)=\tau_\M(h^{p-1}x)$. Thus
  $S^*(S(h)^{p-1})=h^{p-1}$ (here the real or complex adjoint are the
  same because $S$ is self-adjoint preserving).  We set $u=h$ and
  $v=S(h)$, they are both positive of norm 1, because $h^{p-1}=S^*(y)$ for some
  non zero $y\in L_{p'}(\N)$, $h^{-1}\in \M$. 

  Consider the map
  $V_0: \M\to L_p(\N); x\mapsto v^{-\frac 12} S(u^{\frac 12} x
  u^{\frac 12})v^{-\frac 12}$, it is well defined as $v^{-1/2}\in \N$
  and maps positive elements to positive elements.  We have
  $V_0(1_\M)=1_\N$, with positivity, it yields that it actually takes
  values in $\N$ and $\|V_0\|_{\M\to \N}=\|V_0(1)\|=1$.  Similarly,
  $V_1^*: \N\to L_{p'}(\M); y\mapsto u^{-\frac {p-1}2} S^*(v^{\frac
    {p-1}2} y v^{\frac {p-1}2})u^{-\frac {p-1}2}$ is well defined,
  positive with $V_1^*(1_\N)=1_\M$. In particular $V_0$ and $V_1$ have
  norm one.

  Note that for every $x\in \M^+$,
  $0\leq S_0(x)\leq \|\tilde T+U\|_+ V_0(x)$ because
  $0\leq T\leq \tilde T+U$ as positive operators on $L_p(\M)$. One
  concludes that $S_0$ indeed takes values in $\N$ using a Jordan
  decomposition for instance. The last inequality $\|S_0(1)\|\leq \|\tilde T+S\|_+\leq t$ is also clear choosing $\eps$ small enough.

  The arguments for $S_1^*$ are the same. We need to show that $S_1^*$
  is indeed an adjoint. Consider the map
  $W:\M \to L_1(\N); x\mapsto v^{\frac{p-1}2} T(u^{\frac
    {1-p}2}xu^{\frac {1-p}2})v^{\frac{p-1}2}$. It is well defined and continuous by  the H\"older inequality. We need to extend it to 
$L_1$ which reduces to prove that it is continuous for the $L_1$-norm on $\M$. But for 
$x\in\M$ and $y\in \N$, the identity $\tau_\N(W(x)y)=\tau_\M(xS_1^*(y))$ gives $\|W(x)\|_1\leq t \|x\|_1$
because we know that $S_1^*$ is bounded  from $\N$ to $\M$. So $W$ extends to a map $S_1$ whose adjoint is indeed $S_1^*$. The argument to define a pre-adjoint of $S_0$ is the same. 

The completely bounded case is obvious as one just need to check complete positivity of
$S_0:\M\to \N$, $S_1^*:\N\to \M$, this is direct from that of $T:L_p(\M)\to L_p(\N)$
and $T^*:L_{p'}(\N)\to L_{p'}(\M)$ by restriction.
\end{proof}
\begin{rk}
We choose to stick to finite von Neumann algebras to avoid technical problems in the next corollary.  One can extend Theorem \ref{ncschur} to semifinite von Neumann admitting a faithful state using the formulation of \cite{AR} Theorem 5.1.
\end{rk}
The following generalizes a result of Pisier for Schatten classes \cite{Pis95}:
\begin{cor}\label{cb=b}
   $(\M, \tau_\M)$ and $(\N, \tau_\N)$ be semifinite von Neumann algebras, $1\leq p\leq \infty$ and
   $T: L_p(\M)\to L_p(\N)$ be a bounded positive map then  $\|T\|_+=\|T\|$.

   If $T$ is completely positive, then $\|T\|_{cb}=\|T\|=\|T\|_+$.
\end{cor}
\begin{proof} This result is well known if $p=1,\infty$, so we assume $1<p<\infty$. Note that the inequality
  $\|T\|\leq 4\|T\|_+$ is obvious.  Using that $$\|T\|=\sup
  \| fT(e.e)f\|_{L_p(e\M e)\to L_p(f\N f)}$$ where $e$ and $f$ run over finite projections in $\M$ and $\N$,
  and similarly for $\|T\|_+$, one reduces to finite von Neumann algebras. We can assume
  $\|T\|_+=1$ and the algebras are finite.

  The result then follows by complex interpolation, we refer to \cite{BL} and give most of the details. Fix $\eps>0$ and apply Theorem 
  \ref{ncschur} to $T$ and $t=1+\eps$. Let $r_n=1_{u\leq n}$ and $q_n=1_{v\leq n}$.

  As usual, let $\Delta=\{z\in \bC\;|\;  0<{\rm Re }\, z<1\}$. For $z\in \overline \Delta$, let
  $\theta(z)=\frac 12 -\frac {pz} 2$ and define formally for $x\in \M$ with polar decomposition $x=\mu|x|$ and such that 0 is isolated in the spectrum of $|x|$, 
  $$F_x(z) = v^{-\theta(z)}q_n T\big(u^{\theta(z)}r_n \mu |x|^{pz}r_n u^{\theta(z)}\big)v^{-\theta(z)}q_n.$$
  Because of the presence of $q_n$ and $r_n$, the maps $z\mapsto v^{-\theta(z)}q_n, r_n u^{\theta(z)}$ are continuous from $\overline \Delta$ to $\M$ or $\N$ and analytic in $\Delta$.
  Because 0 is isolated in the spectrum of $|x|$, $z\mapsto |x|^{z}$ is also continuous on
  $\overline \Delta$ and analytic in $\Delta$ with values in $\M$.
  As $\M \subset L_p(\M)$ continuously, it follows that $F_x$ is continuous on
  $\overline \Delta$ with values in $L_p(\N)\subset L_1(\N)+\N$ (continuous inclusion) and analytic on $\Delta$.

  We take a look when $z=it$, $t\in \bR$. We have
  $F_x(it)=q_nv^{-\frac{ipt}2}S_0(u^{\frac{ipt}2}r_n\mu |x|^{ipt}r_nu^{\frac{ipt}2})q_nv^{-\frac{ipt}2}$, so $F_x:i\bR\to \N$.
  Again by continuity of $t\mapsto r_nu^{ipt/2}, q_nv^{-ipt/2}$ with values in $\M$ and $\N$,
  $F_x:i\bR\to \N$ is continuous. Moreover $\sup_{t\in \bR}\| F_x(it)\|_\N\leq (1+\eps)$ by continuity of $S_0$.

  When $z=1+it$, $t\in \bR$, we have $F_x(t)=q_nv^{\frac{ipt}2}S_1(u^{-\frac{ipt}2}r_n\mu|x|^{p+ipt}r_nu^{-\frac{ipt}2})q_nv^{\frac{ipt}2}$ (we are exactly in the easy situation of the definition of $W$ in the previous proof).
  Similarly the map $F_x: 1+i\bR\to L_1(\N)$ is continuous and
  $\sup_{t\in \bR}\| F_x(1+it)\|_1\leq (1+\eps)\||x|^p\|_1$ by continuity of $S_1$.

  By complex interpolation $L_p(\N)=(\N,L_1(\N))_{\frac 1p}$, we deduce that $\| F_x(\frac 1p)\|_p\leq (1+\eps) \big(\tau_\M(|x|^p)\big)^{\frac 1p}$. Thus
  $\|q_nT(r_nxr_n)q_n\|_p\leq (1+\eps)\|x\|_p$ and one conclude by taking the limit in $n$ and by density of such $x$ in $L_p$ and because $\eps$ can be arbitrary small.

  The proof for the cb case is the same as $S_0$ and $S_1$ also have completely bounded norms less than $1+\eps$. 
\end{proof}
The following is a direct consequence of $\|T\|=\|T\|_+$.
\begin{cor}\label{monot} Let $S, T :L_p(\M)\to L_p(\N)$ be positive bounded maps with $S\leq T$, then $\|S\|\leq \|T\|$.
  \end{cor}
  
We present a variation of corollary \ref{cb=b} that will be helpful to us.  We assume that
$(\M, \tau_\M)$ and $(\N, \tau_\N)$ are semifinite algebras, $1\leq p\leq \infty$ and $(T_t)$
is a strongly continuous family of positive maps
$\bR\to B(L_p(\M), L_p(\N))$ with compact support $K$ (to simplify)
and which is uniformly bounded ($\sup_t \|T_t\|<\infty$). Then one can
define a generalized convolution map (see \cite{Arendt} Section 1.3)
$$K_T : L_p(\bR;L_p(\M))\to L_p(\bR;L_p(\N)), f\mapsto \Big(t\mapsto \int_K T_{u}\big(f(t-u)\big)du\Big).$$

\begin{cor}\label{convo}
  In the above situation 
  $$\|K_T\|_{L_p(\bR;L_p(\M))\to L_p(\bR;L_p(\N))} \leq \Big\|\int_K T_udu\Big\|_{L_p(\M)\to L_p(\N)}.$$
\end{cor}

\begin{proof}
The result is already known for $p=1,\infty$, so we deal only with $1<p<\infty$.  The proof follows the exact same scheme as for Corollary \ref{cb=b}. One can assume that
  the von Neumann algebras $\M,\,\N$ are finite and restrict to $L_p(I; L_p(\M))\to  L_p(J; L_p(\N))$ for finite intervals $I$ and $J$ to compute the norm.

  Let $V=\int_K T_udu :{L_p(\M)\to L_p(\N)}$ and $\eps>0$, apply
  Theorem \ref{ncschur} to $V$ and $t=1+\eps$ assuming
  $\Big\|\int_K T_udu\Big\|_{L_p(\M)\to L_p(\N)}=1$. There exist
  $u\in L_p(\M)$,  $v\in L_p(\N)$ with norm one and bounded inverses that satisfy the conclusion, we denote by $V_0:\M\to\N$ and $V_1:L_1(\M)\to L_1(\N)$ the resulting maps with norm less than $1+\eps$.  

  Let  $\tilde u\in L_p(I;L_p(\M))=L_p(\A)$ be the function $t\mapsto u$ on $I$ and
  $\tilde v\in L_p(J;L_p(\N))=L_p(\B)$ be the constant function $v$, they have inverse in $\A$ and $\B$.
The algebra $\A$ and $\B$ are finite.
  Now we introduce
   $$ S_0 : \A\to L_p(\B); x\mapsto {\tilde v}^{-\frac 12} K_T({\tilde u}^{\frac 12} x {\tilde u}^{\frac 12}){\tilde v}^{-\frac 12},
   \quad  S_1^* : \B\to L_{p'}(\A); y\mapsto {\tilde u}^{-\frac {p-1}2} K_T^*({\tilde v}^{\frac {p-1}2} y {\tilde v}^{\frac {p-1}2}){\tilde u}^{-\frac {p-1}2}.$$
   They are well defined bounded and positive.

   The unit of $\A$ is the constant function $1_\M$ on $I$.  Because
   of the positivity, we have a.e. for  $t\in J$,
   $$0\leq S_0(1_\A)(t)\leq \int_K {v}^{-\frac {1}2}T_t(u){v}^{-\frac{1}2} dt\leq V_0(1_\M).$$
   In particular, $S_0$ takes values in $\B$ and $\|S_0\|_{\A\to \B}\leq 1+\eps$.
   Similarly $S_1^*:\B\to \A$ has norm less than $1+\eps$ and is the adjoint of a map $S_1:L_1(\A)\to L_1(\B)$ given by  extension of  $x\mapsto \tilde v^{\frac{p-1}2} K_T(\tilde u^{\frac
    {1-p}2}x\tilde u^{\frac {1-p}2})\tilde v^{\frac{p-1}2}$ well defined on $\A$.

   The end of the proof is then by complex interpolation exactly as in Corollary \ref{cb=b} from which we use the
    notation. The element
   $q_n$ and $r_n$ are defined by functional calcul with $1_{[0,n]}$
     from $\tilde v$ and $\tilde u$. Consider $x\in \A$ with 0
     isolated in its spectrum with polar decomposition $x=\mu |x|$, we
     introduce
    $$F_x(z) = \tilde v^{-\theta(z)}q_n K_T\big(\tilde u^{\theta(z)}r_n \mu |x|^{pz}r_n \tilde u^{\theta(z)}\big)\tilde v^{-\theta(z)}q_n.$$
    The rest of the argument to deduce $\|K_T(x)\|_p\leq (1+\eps)\|x\|_p$ is the same using $\|S_0\|_{\A\to \B}\leq 1+\eps$ and   $\|S_1\|_{L_1(\A)\to L_1(\B)}\leq 1+\eps$. One concludes by density of such $x$ in $L_p(I,L_p(\M))$.
  \end{proof}

We leave another variation for the reader, for a completely positive trace
preserving map $U:\A\to \B$, we still denote by $U$ its extension to $L_p$.
\begin{cor}\label{cor2}
  Let $U_i:\A\to \B$, $i=1,...,n$, be completely positive trace preserving maps between
  semifinite von Neumann algebras.
  Let $T_i:L_p(\M) \to L_p(\N)$ be completely positive maps between semifinite von Neumann algebras. If $\sum_{i=1}^n T_i\otimes U_i$ extends to a completely positive map $L_p(\M\overline \otimes \A)\to L_p(\N\overline \otimes \B)$, then
  $$\big\| \sum_{i=1}^n T_i\otimes U_i\big\|_{cb(L_p(\M\overline \otimes \A), L_p(\N\overline \otimes \B))}\leq\big \| \sum_{i=1}^n T_i\big\|_{L_p(\M)\to L_p(\N)}.$$
\end{cor}

When $\A$ and $\B$ are commutative, then one also have that if $T_i$'s are positive bounded then 
$$\big\| \sum_{i=1}^n T_i\otimes U_i\big\|_{L_p(\M\overline \otimes \A) \to L_p(\N\overline \otimes \B)}\leq\big \| \sum_{i=1}^n T_i\big\|_{L_p(\M)\to L_p(\N)}.$$

\section{Estimation of positive Kreiss bounded semigroups}
We start by recalling a basic inequality; it can be found in \cite{DS} Theorem 5.3 with $\Phi(t)=t^p$, $p\geq 1$ (see also Corollary 4.5.3 from \cite{Hiai} when $\M=B(H)$):
\begin{lemma} \label{lpsum}
Let $N>0$, if $(x_i)_{i\in \{0,1,...,N\}} \in L_p(\M)^+$, where $(\M,\tau)$ is semifinite,  then
$$\Big\|\sum_{k=0}^N  x_i\Big\|_p \geq \left( \sum_{k=0}^N \| x_i\|^p_p\right)^{\frac{1}{p}}.$$
\end{lemma}
 
\subsection{Discrete case}

We start by recalling the notions of Kreiss and Cesàro boundednesses in the discrete setting.  


\begin{defi}
 Let $T$ be a bounded operator on a Banach space.
 \begin{enumerate}[(i)]
 \item $T$ is Kreiss bounded (with constant $K_T$) if $\sigma(T)\subset \overline{\mathbb{D}} $ and there exists $K>0$ such that every all $|\lambda|>1$,
$  \|(\lambda I -T)^{-1})\| \leq \frac{K}{|\lambda|-1}$.
The Kreiss constant $K_T$ of $T$ is the smallest possible $K$.
  \item $T$ is Cesàro bounded if there exists $C>0$ such that, for every $N \in \mathbb{N}$, $\Big\| \frac{1}{N+1} \sum_{k=0}^N T^k\big\| \leq C$.
  The Cesàro constant $C_T$ of $T$ is the smallest possible $C$.
 \end{enumerate}
\end{defi}
The following is an extension to non commutative $L_p$ of a classical fact (which also holds for positive operators on a Banach space with a positive normal cone in the sense of \cite{Arendt}):
\begin{prop}\label{kreissces}
Let $T$ be a bounded, positive operator on $L_p(\M)$ with $(\M,\tau)$ semifinite. Then
$T$ is Kreiss bounded iff  $T$ is Cesàro bounded. Moreover, we have $\frac 1 e C_T\leq K_T \leq 4C_T$.
\end{prop}



\begin{proof}
    Let us prove the indirect implication. We suppose that $T$ is Cesàro bounded. In particular, $\|T^n\|\leq (n+1)C_T$ and the sum $\sum_{k=0}^n z^{-k-1} T^k $ converges for $|z|>1$. Its limit is $R(z,T)$ for $|z|>\|T\|$ and thus the limit is the resolvent of $T$ for all $|z|>1$.
    Fix $0\leq  r<1$,   using an Abel transform
       \begin{align*}
        \Big\|\sum_{k=0}^n r^k T^k \Big\| &= \Big\|r^{n+1 }\sum_{k=0}^n T^k + \sum_{k=0}^n (r^k - r^{k+1}) \sum_{j=0}^k T^j \Big\| \\
        & \leq C_T n r^{n+1} + C_T (1 - r) \sum_{k=0}^n r^k k  \leq  \frac{2C_T}{1-r}.
    \end{align*}
        For every $0\leq r< 1$ and $\theta\in \bR$, we have
    $$   - \sum_{k=0}^n r^k T^k  \leq \sum_{k=0}^n T^k r^k\cos(k\theta)\leq  \sum_{k=0}^n  r^kT^k. $$
    Similarly with $\sin$ instead of $\cos$, with $\frac 1 z = re^{i\theta}$, it follows, by Corollary \ref{monot},
$$  \sup_n \Big\| \sum_{k=0}^n z^{-k-1} T^k\Big\|   \leq \frac{4C_T}{|z|-1}.$$ 
    This inequality proves that $T$ is Kreiss.

    We prove the direct implication. We suppose that $T$ is Kreiss bounded. Then $\sigma(T)\subset \overline{\mathbb{D}}$ and it is well known that $z\mapsto R(z,T)$ is holomorphic on $\cc \setminus \sigma(T)$ and can be written with the formula $\sum_{k=0}^{\infty} z^{-k-1} T^k$ if $z\in \cc \setminus \overline{\mathbb{D}}$. 
    Let $n\geq 1$ and choose $z=(1-\frac{1}{n})^{-1}>1$, we get
    \begin{align*}
    (z-1) \sum_{k=0}^{\infty} z^{-k-1} T^k = \frac{1}{n}\sum_{k=0}^{\infty} \left(1-\frac{1}{n}\right)^{k} T^k
    & \geq  \frac{1}{n}\sum_{k=0}^{n-1} \frac{1}{e} T^k\geq 0.
    \end{align*}
     Since the first term is bounded by $K_T$ in norm so is the latter and $T$ is Cesàro bounded with $C_T\leq e K_T$.
    
\end{proof}
It is stated but not proven in \cite{Arnold}, that the norm of $T^n$ for a positive Kreiss bounded operator $T$, on a commutative $L_p$, grows at most like $\Big(\frac{ n }{\log(n)^{\frac{1}{p}}} \Big)$.  We give an improvement for general non commutative $L_p$.

\begin{thm}\label{estdisc}
 Let $T$ be a positive operator on  $L_p(\M)$, with $(\M,\tau)$ semifinite and $1<p<\infty$. If $T$ is Kreiss bounded, then
$$
\|T^n\| \leq (2C_T+1/2)n^{1-\mu_T},
$$
where $\mu_T = \frac{1}{p} \log_4 \left(1 + \frac{1}{(4C_T)^p} \right)$.
\end{thm}

\begin{proof}
  For $(n,s,m) \in \nn^3$, $n\geq s>m\geq 0$,  we have using positivity
  $$  \Big(\sum_{l=m}^{s-1} T^l\Big)\Big(\sum_{k=m}^{s-1}T^{n-k} \Big)\geq
  \sum_{i=0}^{s-m}  (s-m-i) T^{n-i}.$$

  Recall the equivalence between the Kreiss boundedness and the Cesàro
  boundedness of Proposition \ref{kreissces}. Thus,
  $\|\sum_{l=0}^{s-1}T^l\|\leq C_Ts$. The monotony of $\|.\|_p$ on $L_p(\M)^+$  gives, for $x\in L_p(\M)^+$ with $\|x\|_p=1$,  
\begin{align*}
  \frac {s-m}{2s} \Big\|\sum_{i=0}^{\frac {s-m}2} T^{n-i}(x) \Big\|_p \leq\Big\| \sum_{i=0}^{s-m} \frac {s-m-i}{s}T^{n-i}(x) \Big\|_p \leq C_T \Big\|\sum_{k=m}^{s-1}T^{n-k}(x) \Big\|_p.
\end{align*}

In particular, if $m=2l$ and $s=4l$, the latter inequality reads
\begin{align}\label{aa1}
\Big\|\sum_{i=0}^{l} T^{n-i}(x) \Big\|_p \leq 4C_T \Big\|\sum_{k=2l}^{4l-1}T^{n-k}(x) \Big\|_p.
\end{align}
Let $k\in \mathbb{N}$ be such that $4^k \leq n < 4^{k+1}$ and denote $u_l = \|\sum_{i=0}^{l-1}T^{n-i}\|_p^p$ for $1 \leq l \leq n$. For every $l \in \{1, \ldots, k\}$, we have, by \eqref{aa1} and lemma \ref{lpsum},
\begin{equation}
\begin{aligned}\label{aa2}
u_{4^l} \geq \Big\|\sum_{i=0}^{4^{l-1}-1} T^{n-i}(x) +  \sum_{i=2.4^{l-1}}^{4^{l}-1} T^{n-i}(x) \Big\|_p^p
& \geq \Big\|\sum_{i=0}^{4^{l-1}-1} T^{n-i}(x)\Big\|_p^p +  \Big\| \sum_{i=2.4^{l-1}}^{4^{l}-1} T^{n-i}(x) \Big\|_p^p\\
& \geq \left(1 + \frac{1}{(4C_T)^p}\right) u_{4^{l-1}}.
\end{aligned}
\end{equation}
By successively applying inequality \eqref{aa2}, we get
\begin{align*}
u_{4^k} \geq \left(1 + \frac{1}{(4C_T)^p}\right)^k u_1 = \left(1 + \frac{1}{(4C_T)^p} \right)^k \| T^n(x) \|_p^p = 4^{k \delta_T} \| T^n(x) \|_p^p,
\end{align*}
where $\delta_T := \log_4 \left(1 + \frac{1}{(4C_T)^p} \right)$. On one hand,
$$
u_{4^k} \leq \Big\|\sum_{i=0}^{n} T^{n-i}(x)\Big\|_p^p \leq C_T^p (n+1)^p \leq (2C_T)^p n^p,
$$
and on the other hand, since $4^k \geq \frac{n}{4}$,
$$
4^{k \delta_T} \geq \frac{n^{\delta_T}}{4^{\delta_T}} = \dfrac{n^{\delta_T}}{\left(1 + \frac{1}{(4C_T)^p} \right) }.
$$
It follows that
$$
\| T^n(x) \|_p^p \leq \left(1 + \frac{1}{(4C_T)^p} \right) (2C_T)^p n^{p-\delta_T} = \left( 2^pC_T^p + \frac{1}{2^p} \right) n^{p-\delta_T},
$$
which concludes the proof.

\end{proof}

\begin{rk}
Actually one can replace $\|\sum_{k=a}^b T^{n-k}(x)\|_p$ by $\Big(\sum_{k=a}^b \|T^{n-k}(x)\|_p^p\Big)^{1/p}$. Indeed by Corollary \ref{cor2} applied to $T_i=T^i$ and $S_i=S^i$ where $S$ is the shift on $\ell_\infty(\mathbb Z)$ gives that $T\otimes S$ has the same Kreiss constant as $T$. One just needs to apply the above proof with $T\otimes S$ and $x$ replaced by $x\otimes \delta_0$. 
\end{rk}

\begin{thm}\label{estdischil}
Let $T$ be an operator on a Hilbert space $H$. If $T$ is Kreiss bounded, then there is some $\delta_T>0$ such that $\|T^n\| =\mathcal{O}\big(n^{1-\delta_T}\big)$.
\end{thm}

\begin{proof}
For every $z\in \mathbb{T}$, we let $ A(z) :=\sum_{k=0}^{\infty} z^k r^k T^k$ and $B(z) :=\sum_{i=m}^{s-1} z^{n-i} T^{n-i} $ for some fixed integers $n\geq s> m\geq 0$ and $r=1-\frac{1}{s}$. As in the proof of Proposition \ref{kreissces}, the operator $A(z)$ is well defined, satisfies $A(z)= (zr)^{-1} R\left((zr)^{-1},T\right)$, and the Kreiss boundedness of $T$ yields 
\begin{equation}\label{estdisK1}
\|A(z)\| \leq \frac{K_T}{1-r} = sK_T.
\end{equation}
Let $C(z) = A(z)B(z)$ and notice that $\displaystyle{C(z) = \sum_{l=n-s+1}^{\infty} C_l z^l T^l}$ with $\displaystyle{C_l = \sum_{\substack{l=n+k-i \\ m\leq i \leq s-1 \\ k \geq 0}} r^k}$. In particular, when $l=n-j$ where $0 \leq j \leq s-m$, we have
\begin{equation}\label{estdisK2}
C_{n-j}  = \sum_{\substack{i=k+j \\ m\leq i \leq s-1 \\ k \geq 0}} r^k \geq \sum_{\substack{i=k+j \\ m\leq i \leq s-1 \\ m \leq k \leq s-1}} \big(1-1/s \big)^{s-1} \geq e^{-1}(s-m-j).
\end{equation}
Let $x\in H, \|x\|=1$. On one hand, by \eqref{estdisK1} we have
$$
\| z \mapsto A(z)B(z)x \|^2_{L_2(\mathbb{T}, H)} \leq s^2K_T^2 \| z \mapsto B(z)x \|^2_{L_2(\mathbb{T}, H)}= s^2K_T^2 \sum_{i=m}^{s-1} \|T^{n-i} x\|^2.
$$
On the other hand, by \eqref{estdisK2},
\begin{align*}
\| z \mapsto A(z)B(z)x \|^2_{L_2(\mathbb{T}, H)}
& = \| z \mapsto C(z)x \|^2_{L_2(\mathbb{T}, H)} \\
& = \sum_{l=n-s+1}^{\infty} C_l^2\|T^lx\|^2 \\
& \geq e^{-2} \sum_{j=0}^{s-m} (s-m-j)^2 \|T^{n-j} x\|^2.
\end{align*}
It follows that
\begin{equation}\label{estdisK3}
\frac{e^{-2}}{4} \frac{(s-m)^2}{s^2} \sum_{j=0}^{\frac{s-m}{2}}  \|T^{n-j} x\|^2 \leq \frac{e^{-2}}{s^2} \sum_{j=0}^{s-m} (s-m-j)^2 \|T^{n-j} x\|^2 \leq K_T^2 \sum_{i=m}^{s-1} \|T^{n-i} x\|^2.
\end{equation}
Moreover, the first lines of the proof of \cite[Theorem 4.1]{CCEL} show that for every $N\in \mathbb{N}$,
\begin{equation}\label{estdisK4}
\sum_{i=0}^{N} \|T^{i} x\|^2 \leq e^2K_T^2 (N+1)^2.
\end{equation}
Combining the inequalities \eqref{estdisK3} and \eqref{estdisK4}, and arguing as in the proof of Theorem \ref{estdisc}, we obtain the desired estimate.
\end{proof}

\subsection{Continuous case}

We refer to \cite{Arendt} for the general theory of semigroup.
\begin{defi}
 Let  $T=(T_t)_{t\geq0}$ be a  $C_0$-semigroup on a Banach space and let $A$ be its generator.
 \begin{enumerate}[(i)]
 \item $T$ is Kreiss bounded if there exists $K>0$ such that, for every $\lambda\in\cc$, $Re(\lambda)>0$ we have $\lambda \in \rho(A)$ and
$$\|R(\lambda,A)\|\leq \frac{K}{{\rm Re\,}(\lambda)}.$$
The smallest possible $K$ is the Kreiss constant $K_T$ of the semigroup.
  \item $T$ is Cesàro bounded if there exists $C>0$ such that, for every $x\in X$ and every $t>0$ we have
  $$\Big\|\frac{1}{t}\int_0^t T_sxds\Big\|\leq C\|x\|.$$
The smallest possible $C$ is the  Cesàro constant $C_T$ of the semigroup.
 \end{enumerate}
\end{defi}

For positive and Kreiss bounded semigroups on commutative $L_p$-spaces, Arnold and Coine in \cite{Arnold} obtained the estimate $\|T_t\|= \O \left( \frac{t}{\log(t)^{\max\{1/p,1/p'\}}} \right)$. This has been recently improve to
$\|T_t\|= \O \left( t^{1-\varepsilon_T}\right)$
for commutative $L_p$-spaces in \cite{Arnold3}.
As before, we extend it for non commutative $L_p$. The scheme is the same as in the discrete case.

\begin{prop}\label{kreisscescont}
Let $(T_t)_{t\geq 0}$ be a positive $C_0$-semigroup on $L_p(\M)$ with $(\M,\tau)$ semifinite. Then
$(T_t)_{t\geq 0}$ is Kreiss bounded iff  $T$ is Cesàro bounded. Moreover, we have $\frac 1 e C_T\leq K_T \leq 4C_T$.
\end{prop}
\begin{proof}
 The proof is basically the same as Proposition \ref{kreissces}. It is also similar to Proposition 2.6 \cite{Arnold} for commutative $L_p$. The only point is to justify that the resolvent formula $R(\lambda,T)=\lim_{T\to \infty} \int_0^T e^{-\lambda s} T_sds $ holds in $\{\lambda\in \bC \mid {\rm Re}\,\lambda>0\}$. This is exactly Theorem 5.3.1 in \cite{Arendt} because non commutative $L_p$ are Banach spaces with normal positive cone.
 We skip the details.
\end{proof}
 
\begin{thm}\label{estcont}
Let $1<p<\infty$ and let $(T_t)_{t\geq 0}$ be a positive Kreiss bounded $C_0$-semigroup on $L_p(\M)$ with $(\M,\tau)$ semifinite.  Then there is some $\delta_T>0$ such that $\|T_t\| =\mathcal{O}\big(t^{1-\delta_T}\big)$.
\end{thm}
\begin{proof}
    The proof starts as for Theorem \ref{estdisc}. Replacing $\sum_{l=m}^{s-1} T^l$ by $\int_{m}^s T_udu$, one ends up with 
    $$\Big\| \int_{t-1}^t T_u du\Big\|= \mathcal{O}\big(t^{1-\delta_T}\big).$$
for some $\delta_T>0$. 

Now we consider the convolution with the strongly continuous operator valued function $k_t: u\mapsto 1_{[t-2,t]}(u) T_u$ (see \cite{Arendt} Section 1.3). It yields a positive continuous map  $K_t:L_p(\bR;L_p(\M))\to L_p(\bR;L_p(\M))$ because $(T_u)_{u\geq 0}$ is positive. Using Corollaries \ref{convo} and \ref{monot} and the abvove bound, 
$$\|K_t\| \leq \Big\|\int_{t-2}^t T_u du\Big\| =\mathcal{O}\big(t^{1-\delta_T}\big).$$
Consider $x\in L_p(\M)^+$ with norm one and $f(u)=1_{[0,1]}T_u(x)$. Then 
$\|f\|_p\leq \sup_{u\in [0,1]} \|T_u\|=M_T$ and $K_t(f)(t-s)= \int_0^1 k_t(t-s-u)T_{u}(x)=T_{t-s}(x)$ for $s\in [0,1]$. It gives
$$ \int_0^1 \|T_{t-s}(x)\|_p^p ds\leq \|K\|^p M_T^p.$$
By strong continuity there is some $s\in[0,1]$ with $\|T_{t-s}(x)\|_p\leq \|K_t\|M_T$ and thus  $\|T_{t}(x)\|_p\leq \|K_t\|M_T^2$. This concludes the proof.
\end{proof}

As in the previous section, for Hilbert spaces, one can remove the positivity assumption, giving an alternative but close  proof to \cite{Arnold3} :
\begin{thm}\label{estconth}
Let $(T_t)_{t\geq 0}$ be a  Kreiss bounded $C_0$-semigroup on a Hilbert space $\mathcal H$.  Then there is some $\delta_T>0$ such that $\|T_t\| =\mathcal{O}\big( t^{1-\delta_T}\big)$.
\end{thm}
\begin{proof}
    It is already known  \cite{Arnold2} that $\|T_t\| =o\big(t\big)$. It follows that the strongly continuous map
    $\bR\to B(\mathcal H); y\mapsto \chi_{y>0}e^{-yr} T_y$ has $A: y\mapsto R(r+iy)$ as Fourier transform for $r>0$. As $T$ is Kreiss $\sup_{y\in \bR} \|A(y)\|\leq \frac {K_T}{r}$. 
    
    Fix $x\in\mathcal H$, let $0<m<s<t$ and set $B(y)=\int_m^s 
    e^{-iy(t-u)}T_{t-u}(x)du$. By the Plancherel theorem $\|B\|_{L_2(\bR;\mathcal H)}^2=2\pi\int_m^s \|T_{t-u}(x)\|^2du$. Using an easy change of variable, we have
    $$A(y)B(y)=\int_0^\infty e^{-iyz}\left(\int_m^s e^{-r(z+u-t)}\chi_{z+u\geq t}  du\right)T_z(x) dz:=\int_0^\infty e^{-iyz} \gamma(z)T_z(x) dz .$$ 
    The Plancherel theorem and the previous estimates give
    $$\int_0^{\frac {s-m}2}  |\gamma(t-v)|^2\|T_{t-v}(x)\|^2 dv\leq\int_0^\infty  |\gamma(z)|^2\|T_z(x)\|^2 dz\leq\frac {K_T^2}{r^2} \int_m^s \|T_{t-u}(x)\|^2du.$$
    But if we take  $r=\frac 1 s$, when $0<v<\frac {s-m}2$, then $|\gamma(t-v)|\geq \frac 1 {e}(s-\max\{m,v\})\geq \frac{s-m}{2e}$. Hence
    $$\frac {(s-m)^2}{4e^2}\int_0^{\frac {s-m}2}  \|T_{t-v}(x)\|^2 dv\leq\int_0^\infty  |\gamma(z)|^2\|T_z(x)\|^2 dz\leq s^2K_T^2 \int_m^s \|T_{t-u}(x)\|^2du.$$
    Iterating as in Theorem \ref{estdisc}, we end up with
    $$\int_{t-1}^t \| T_u(x)\|^2 du\lesssim_{T} t^{2(1-\delta_T)} \|x\|^2.$$
    And one concludes as in the end of the proof of Theorem \ref{estcont}.
\end{proof}
\subsection{Spectral bounds} Let $T=(T_t)_{t\geq 0}$ be a $C_0$-semigroup with generator $A$. Recall the definitions for the spectral bound
$$s(A)=\sup\{ {\rm Re\,} \lambda ; \;\lambda\in \sigma(A)\}$$
and the growth bound
$$\omega(T)=\inf\{ \omega\in \bR\, |\, \sup_{t>0} e^{-t\omega} \|T_t\|<\infty\}.$$
We show below that for positive semigroups on a non commutative $L_p$-space, $s(A)=\omega(T)$.
Thanks to Corollary \ref{convo}, the proof of \cite{Wei98} in the commutative setting can be readily adapted to establish the same equality in the non commutative setting.
\begin{thm}
Let $T=(T_t)_{t\geq 0}$ be a positive $C_0$-semigroup on $L_p(\M)$ where $(\M,\tau)$ is semifinite and $1\leq p<\infty$, then $s(A)=\omega(T)$. 
\end{thm}
\begin{proof}
It is well known that $s(A)\leq \omega(T)$ holds for all semigroups
(Theorem 5.1.9 in \cite{Arendt}). 
By a change of semigroup, it suffices to show that $\omega(T)<0$ if $s(A)<0$.

  From Theorem 5.3.1 in \cite{Arendt}, if $s(A)<0$, then $$\sup_{T>0} \Big\|\int_0^T T_t dt\Big\|:= M<\infty.$$ 
  Consider for $x\in L_p(\M)$, $f\in L_p(\bR^+;L_p(\M))$ given by $f(t)=\chi_{t<1}T_t(x)$. Using convolution with the operator kernel $K(t)=\chi_{0<t<T} T_t$, Corollary \ref{convo} gives
  $$\int_1^{T-1} \|T_t(x)\|_p^p dt \leq M^p \int_0^1  \|T_t(x)\|_p^p dt.$$
  Letting $T\to \infty$, we get that $t\mapsto T_t(x)\in L_p(\bR^+;L_p(\M))$. Thus by Datko's theorem (Theorem 5.1.2 in \cite{Arendt}), $\omega(T)<0$.
  \end{proof}

\section{Optimal examples}

\subsection{Some weights}

\begin{notation}
Let $\omega$ be a weight over an interval $I$, i.e. non negative and locally integrable. We denote for $1< p<\infty$
\begin{align}\label{Ap}[\omega]_{\{A_p,I\}} = \sup_{\{(y,z)\in I, y<z\}} \frac{1}{(z-y)^p} \int_y^{z} \omega(x)dx\Big( \int_y^{z} \omega(x)^{-\frac 1{p-1}}dx\Big)^{p-1}. \end{align}
If $[\omega]_{\{A_p,I\}}<\infty$, we will say that $\omega$ satisfies the $A_p$-condition on $I$.
\end{notation}

This corresponds to the Muckenhoupt condition on $I$. We have the following useful easy dilation invariance.
\begin{lemma}\label{inv}
If a weight $\omega$ satisfies the $A_p$-condition on the interval $[a,b]$ then any dilation $\omega_c:x\mapsto \omega(x/c)$, $(c>0)$ also satisfies it on $[ca,cb]$ and $[\omega]_{\{A_p,[a,b]\}}=[\omega_c]_{\{A_p,[ca,cb]\}}$.
\end{lemma}

From now on, we consider the translation semigroup $T$ on $I$  given by for $f\in L_p(I,\omega)$, $x\in I$ 
$$ T_{t} (f)(x)=\left\{
\begin{array}{ll}
f(x+t) & \textrm{ if } x+t\in I\\
0 & \textrm{ if }x+t\notin I.
\end{array}\right.$$
The following is well known when $I=\bR$ or in the discrete case (see for instance \cite{MTX}). We provide an easy proof using the Schur test.  
\begin{prop} If $\omega$ satisfies the $A_p$-condition on $I$, then $(T_t)_{t\geq 0}$ is Cesàro bounded and $C_T\leq 2 [\omega]_{\{A_p,I\}}^{1/p}$.
\end{prop}
\begin{proof}
We consider the standard isometric bijection 
$$\iota: L_p(I,\omega)\to L_p(I), f\mapsto f{\omega^{1/p}}, \qquad \kappa: L_p(I)\to L_p(I,\omega), f\mapsto f\omega^{-1/p}.$$

 The norm of the Ces\`aro  mean, $M_t=\frac 1 t\int_0^t T_u du$ on $L_p(I,\omega)$ with $t>0$, is the same as that of $\tilde M_t=\iota M_t\kappa$ on $L_p(I)$. The kernel of $\tilde M_t$ is 
 $K_t(x,y)= \frac{1}{t} \chi_{x\leq y\leq x+t} \frac{\omega(x)^{1/p}}{{\omega(y)^{1/p}}}$. To compute its norm, we use the Schur test Proposition \ref{schurtest}. We choose for $x,y\in I$ 
 $$A(x,y) = \chi_{x\leq y\leq x+t} \omega(y)^{-\frac 1{p-1}} \gamma(x)^{-1},\qquad B(x,y)=t^{-p} \chi_{x\leq y\leq x+t} \omega(x)\gamma(x)^{p-1}$$ where
$$\gamma(x) =  \int_{x}^{x+t} \omega(y)^{-\frac 1{p-1}}dy.$$
Of course, we extend functions outside $I$ by 0. First, we have
$$    \int_I  A(x,y)dy = \gamma(x)^{-1}\int_{x}^{x+t}  \omega(y)^{-\frac 1  {p-1}}dy= 1.
$$
And
\begin{align*}
     \int_I B(x,y)dx  &=  t^{-p} \int_{y-t}^{y} {\omega(x)} \Big(\int_{x}^{x+t}  \omega(z)^{-\frac 1{p-1}}dz\Big)^{p-1}dx\\
    & \leq t^{-p}  \int_{y-t}^{y} {\omega(x)} \Big(\int_{y-t}^{y+t}  \omega(z)^{-\frac 1{p-1}}dz\Big)^{p-1}dy\\
    &\leq t^{-p} \big( \int_{y-t}^{y+t} {\omega(x)}dy \big) \big( \int_{y-t}^{y+t}  \omega(z)^{-\frac 1{p-1}}dz \big)^{p-1} \\
    & \leq 2^p[\omega]_{\{A_p,I\}}.
\end{align*}
The estimate from the Schur test gives the conclusion.
\end{proof}
\begin{rk}
The point here is that we have a power $\frac 1p$ rather than $\frac 1{p-1}$ which would follow from the boundedness of maximal Hardy-Littlewood operator (when $I=\bR)$,  see \cite{Gra2}.
\end{rk}

\begin{rk}
When the weight is decreasing, we also have an $2C_T\geq [\omega]_{\{A_p,I\}}^{1/p}$.

Indeed, $T_{-t}$ is then a contraction and 
$\frac 1 {t} \int_{-t}^{t} T_udu$ has kernel $K_t(x,y)=\frac 1 t \chi_{|x-y|<t}$.
For fixed $y$ with $z=y+t$, $K_t(a,b)\geq \frac{1}{t} 1_{(y,z)}(a)1_{(y,z)}(b)$. The corresponding integral operator is of rank 1 with norm $\big( \int_{y}^{z} {\omega(u)}du \big)^{1/p} \big( \int_{y}^{z}  \omega(u)^{-\frac 1{p-1}}dz \big)^{1- 1/p}$, thus 
$$ [\omega]_{\{A_p,I\}}^{1/p}\leq C_T+ 1\leq 2 C_T.$$

\end{rk}
\begin{notation}
For $\gamma\in (0,1]$ and $1<p<\infty$, let us define $\omega_\gamma$ as a positive weight on $]-1,1[$ such that $$\omega_\gamma(x) = \left\{
\begin{array}{rr}
(x+1)^{-\gamma},& x\leq 0\\
(1-x)^{\gamma(p-1)},& x>0.
\end{array}
\right. $$
\end{notation}
We do not emphasize that $\omega_\gamma$ depends on $p$.

\begin{lemma}\label{estAp}
For $0<\gamma<1$ and $1<p<\infty$, we have,
$$[\omega_\gamma]_{\{A_p,]-1,1[\}}\leq  \left(  \frac{2}{1-\gamma}  \right)^{p}.$$
\end{lemma}
\begin{proof} Set $I=]-1,1[$. 
    To prove it, we will compute the supremum \eqref{Ap} over $y,z\in I$, $y<z$ according to the three cases $z<0$, $y>0$ and $y<0<z$. 
    We also set $t=z-y>0$.

{\it Case 1}. By a direct computation, we have
\begin{align*}
  &\frac{1}{t^p} \int_y^{y+t} (x+1)^{-\gamma}dx \left(\int_y^{y+t} (1+x)^{\frac{\gamma}{p-1}}dx\right)^{p-1} =\\
      &  \left[ \frac{1}{1-\gamma} \left( \frac{(y+t+1)^{1-\gamma}}{t^{1-\gamma}} - \frac{(y+1)^{1-\gamma}}{t^{1-\gamma}} \right) \left(\frac{1}{1+\frac{\gamma}{p-1}}\right)^{p-1} \left( \frac{(y+t+1)^{1+\frac{\gamma}{p-1}} - (y+1)^{1+\frac{\gamma}{p-1}}}{t^{1+\frac{\gamma}{p-1}}} \right)^{p-1} \right]
\end{align*}
Let $b=\frac{y+1}{t}$, $b\in]0,+\infty[$, so that 
\begin{align*}
     &\frac{1}{t^p}   \left( (y+t+1)^{1-\gamma} - (y+1)^{1-\gamma} \right)  \left( (y+t+1)^{1+\frac{\gamma}{p-1}} - (y+1)^{1+\frac{\gamma}{p-1}} \right)^{p-1} \\
    & = \left( (1+b)^{1-\gamma} - b^{1-\gamma} \right)\left( (1+b)^{1+\frac{\gamma}{p-1}} -b^{1+\frac{\gamma}{p-1}}  \right)^{p-1}.
\end{align*}
If $b\leq1$, by concavity and convexity, we have 
\begin{align*}
    \left( (1+b)^{1-\gamma} - b^{1-\gamma} \right) \leq 1, \quad 
    & \left( (1+b)^{1+\frac{\gamma}{p-1}} -b^{1+\frac{\gamma}{p-1}}  \right)^{p-1} \leq  \Big(1+\frac{\gamma}{p-1}\Big)^{p-1} 2^\gamma
\end{align*}
and if $b\geq1$
\begin{align*}
    \left( (1+b)^{1-\gamma} - b^{1-\gamma} \right) \leq (1-\gamma)b^{-\gamma}, \quad 
    & \left( (1+b)^{1+\frac{\gamma}{p-1}} -b^{1+\frac{\gamma}{p-1}}  \right)^{p-1} \leq \left( 1+\frac{\gamma}{p-1}  \right)^{p-1}(b+1)^{\gamma}.
\end{align*}
Hence
\begin{align*}
&\sup_{ (y,z)\in I, y<z<0 }\frac{1}{(z-y)^p} \int_y^{z} \omega_\gamma(x)dx \left(\int_y^{z} {\omega_\gamma(x)}^{\frac{-1}{p-1}}dx\right)^{p-1} \leq \frac{2^{\gamma}}{1-\gamma}.
\end{align*}

{\it Case 2: } Similar computations can be performed on $[0,1[$. We have for $0<y<z=y+t<1$,
\begin{align*}
&\frac{1}{t^p}\int_y^{y+t}(1-x)^{\gamma(p-1)} dx
\left(\int_y^{y+t}(1-x)^{-\gamma} dx\right)^{p-1}\\
&=
\frac{1}{1+\gamma(p-1)}
\left((1+b)^{1+\gamma(p-1)}-b^{1+\gamma(p-1)}\right)
\left[
\frac{1}{1-\gamma}
\left((1+b)^{1-\gamma}-b^{1-\gamma}\right)
\right]^{p-1},
\end{align*}
where
$
b=\frac{1-y-t}{t}>0.
$
If $b\leq 1$, by convexity and concavity,
$$
(1+b)^{1+\gamma(p-1)}-b^{1+\gamma(p-1)}
\leq (1+\gamma(p-1))2^{\gamma(p-1)}, \quad (1+b)^{1-\gamma}-b^{1-\gamma}\leq 1.
$$
Hence the above quantity is bounded by
$$
\frac{2^{\gamma(p-1)}}{(1-\gamma)^{p-1}}.
$$
If $b\geq1$, again by convexity and concavity,
$$
(1+b)^{1+\gamma(p-1)}-b^{1+\gamma(p-1)}
\leq (1+\gamma(p-1))(1+b)^{\gamma(p-1)}, \quad (1+b)^{1-\gamma}-b^{1-\gamma}
\leq (1-\gamma)b^{-\gamma}.
$$
Thus the above quantity is bounded by
$$
\left(\frac{1+b}{b}\right)^{\gamma(p-1)}
\leq 2^{\gamma(p-1)}.
$$
Consequently,
$$
\sup_{0<y<z<1}
\frac{1}{(z-y)^p}
\int_y^z\omega_\gamma(x),dx
\left(
\int_y^z\omega_\gamma(x)^{-\frac1{p-1}},dx
\right)^{p-1}
\leq
\frac{2^{\gamma(p-1)}}{(1-\gamma)^{p-1}}.
$$



{\it Case 3.} Let us take $-1<y<0<z<1$ and $t=z-y$ and compute
\begin{align*}
   &\frac{1}{t^p}  \int_y^{z} \omega_\gamma(x)dx \left(\int_y^{z} {\omega_\gamma(x)}^{\frac{-1}{p-1}}dx\right)^{p-1} =\\  
   &\frac{1}{t^p}  \left(\int_y^{0}(1+x)^{-\gamma} dx + \int_0^{z} (1-x)^{(p-1)\gamma} dx\right) \left(\int_y^{0} (1+x)^{\frac{\gamma}{p-1}} dx +\int_0^{z} (1-x)^{-\gamma} dx\right)^{p-1}.
\end{align*}
Denote $\Tilde{t}=\max\{-y,z\}$. We have $1+y\geq 1-\Tilde{t}$, $1-z \geq 1-\Tilde{t}$ and lastly, since $t=z-y \geq \Tilde{t}$, $\frac{1}{t^p} \leq \frac{1}{\Tilde{t}^p}$. Hence
\begin{align*}
    &\frac{1}{t^p}  \int_y^{z} \omega_\gamma(x)dx \left(\int_y^{z} {\omega_\gamma(x)}^{\frac{-1}{p-1}}dx\right)^{p-1}\leq\\
    & \frac{1}{\Tilde{t}^p}  \left(\int_{-\Tilde{t}}^{0}(1+x)^{-\gamma} dx + \int_0^{\Tilde{t}} (1-x)^{(p-1)\gamma} dx\right) \left(\int_{-\Tilde{t}}^{0} (1+x)^{\frac{\gamma}{p-1}} dx +\int_0^{\Tilde{t}} (1-x)^{-\gamma} dx\right)^{p-1}\\
    &= \left( \frac{1}{1-\gamma} \frac{I_1}{\Tilde{t}} + \frac{1}{1+\gamma(p-1)} \frac{I_2}{\Tilde{t}} \right) \left( \frac{1}{1+\frac{\gamma}{p-1}}\frac{I_3}{\Tilde{t}} + \frac{1}{1-\gamma} \frac{I_4}{\Tilde{t}} \right)^{p-1} :=A,
\end{align*}
where
\begin{align*}
    I_1 =(1-(1-\Tilde{t})^{1-\gamma}), \quad
    &I_2= (1-(1-\Tilde{t})^{1+\gamma(p-1)}) \\
    I_3= (1-(1-\Tilde{t})^{1+\frac{\gamma}{p-1}}), \quad
   &I_4=(1-(1-\Tilde{t})^{1-\gamma}).
\end{align*}
By concavity and convexity, we have
\begin{align*}
    \frac{I_1}{\Tilde{t}} \leq 1, \quad
    &\frac{I_2}{\Tilde{t}} \leq 1+\gamma(p-1) \\
    \frac{I_3}{\Tilde{t}} \leq 1+\frac{\gamma}{p-1}, \quad
   &\frac{I_4}{\Tilde{t}}\leq 1.
\end{align*}
And at the end we have, 
$$A\leq \left( \frac{1}{1-\gamma}  + 1 \right) \left( 1 + \frac{1}{1-\gamma}  \right)^{p-1} = \left( 1 + \frac{1}{1-\gamma}  \right)^{p}\leq \left(\frac{2}{1-\gamma}  \right)^{p}.$$
One concludes as $\frac{2^{\gamma}}{1-\gamma}$ and $\frac{2^{\gamma(p-1)}}{(1-\gamma)^{p-1}}$ are majorized by $\left(\frac{2}{1-\gamma}  \right)^{p}$.
\end{proof}
The case $\gamma=1$ is handled similarly but on smaller intervals. We leave the details to the reader.
\begin{lemma}\label{gam1}
Let $1<p<\infty$ and $N\geq 10$. Then, 
$$[\omega_1]_{\{A_p,]-1+\frac{1}{N},1-\frac{1}{N}[\}}\leq \big(1+2\ln(N)\big)^p.$$ 
\end{lemma}

\subsection{Examples}

We will construct our examples thanks to translations using the weights $\omega_\gamma$ restricted to $\Big[-1+\frac 1 N,1-\frac 1 N\Big]$
by a dilation by a factor $\frac{N-1}{N^2}$ and a translation by $N$, that is $\omega_\gamma^N(x)=\omega_\gamma\Big(\frac {N-1}{N^2} (x-N)\Big)$. In all this section we fix $1<p<\infty$.

\begin{notation}
We consider the $L_p$-space   $B_p^N=L_p( [0,2N],\omega_\gamma^N)$, where $$\omega_\gamma^N(x) = \left\{
\begin{array}{ll}
{N^{\gamma}}{\big(x(1-\frac{1}{N}\big)+1)^{-\gamma}}, & 0\leq x\leq N\\
{N^{-\gamma(p-1)}}{\big(2N-1-x(1-\frac{1}{N})\big)^{\gamma(p-1)}},& N<x\leq 2N.
\end{array}
\right. $$
The left translation semigroup on  $B_N$ is denoted by $(T_{t}^{\gamma,N})_{t\geq 0}$, for $f\in B_p^N$ and $x\in (0,2N)$:
$$T_{t}^{\gamma,N}(f)(x)=\left\{
\begin{array}{ll}
f(x+t),&  x+t\leq 2N\\
0,& x+t>2N.
\end{array}
\right.$$ 
\end{notation}

\begin{notation}
 We consider the $L_p$-space  $B=\ell_p\Big((B_p^N)_{N\geq 2}\Big)$ with the direct sum semigroup
 $(T^\gamma_t)_{t\geq0}:B\rightarrow B$ defined by $T_t^\gamma \Big((f_N)_{N\geq2}\Big) = \Big(T_{t}^{\gamma,N} (f_N)\Big)_{N\geq2}$.
\end{notation}
The main result of this section is
\begin{thm}\label{exam}
For $0<\gamma<1$, $(T_t^\gamma)_{t\geq 0}$ is a positive $C_0$-semigroup on a commutative $L_p$-space that is Kreiss bounded with $\|T^\gamma_t\|\geq K_\gamma  t^{\gamma}$ for some $K_\gamma>0$.
\end{thm}

The proof will will consist in verifying all facts. We start with norm estimates.
\begin{lemma}
Let $N\geq 2$ and $0<\gamma\leq 1$, we have the estimates
\begin{enumerate}[(i)]
\item $ \forall t\in(2N-2;2N)$,
$\big\|T^{\gamma,N}_t \big\|\geq \left( \frac{N}{3} \right)^{\gamma},$
\item $\forall t \in(0,1) $ ,
$\big\|T^{\gamma,N}_t\big\| \leq 2^{\gamma}.$
\end{enumerate}
\end{lemma}
\begin{proof}
    Since the weights are continuous, it is classical that $\big\|T^{\gamma,N}_t \big\|^p=\sup_{x\in(t,2N)}\frac{\omega_\gamma^N(x-t)}{\omega^N_\gamma(x)}.$
    
    For $(i)$, since $t\in (2N-2,2N)$, for $x>t$  and $N\geq 2$, we have that $x-t\leq N$ and $x> N$ so
        \begin{align*}
        \frac{\omega^N_\gamma(x-t)}{\omega^N_\gamma(x)}
        = \frac{N^{\gamma p}}{(x-t+1)^{\gamma} (2N+1-x)^{\gamma(p-1)} } \geq
         \frac{N^{\gamma p}}{3^{\gamma} 3^{\gamma(p-1)} } = \frac{N^{\gamma p}}{3^{\gamma p} }.
    \end{align*}
    
    For $(ii)$, let $0 < t < 1$. We estimate according to the cases $x<N$, $x-t>N$ and $x-t\leq N\leq x$.
    
    In the first case, we have 
        \begin{align*}
       \frac{\omega^N_\gamma(x-t)}{\omega^N_\gamma(x)}  =\frac{(x(1-\frac{1}{N})+1)^{\gamma}}{((x-t)(1-\frac{1}{N})+1)^{\gamma}} = \frac{(u+1)^{\gamma}}{ (v+1)^{\gamma} } \leq 2^{\gamma},
    \end{align*}
  as $v\leq u\leq v+1$  distinguishing $u>1$ and $u\leq 1$.
  
    In the second, we have by symmetry with $y=2N-x$
        \begin{align*}
      \frac{\omega^N_\gamma(x-t)}{\omega^N_\gamma(x)}  =  
          \frac{(1+ (1-\frac 1 N)(y-t))^{\gamma(p-1)}}{(1+(1-\frac 1 N)y)^{\gamma(p-1) }} 
         \leq 2^{\gamma(p-1)}.
    \end{align*}
    
    In the third case, as $x-t\geq N-1$ and $x\leq N+1$, we get
    \begin{align*}
        \frac{\omega^N_\gamma(x-t)}{\omega^N_\gamma(x)} \leq \frac{\omega^N_\gamma(N-1)}{\omega^N_\gamma(N+1)}=\Big(1-\frac 1 N +\frac 1 {N^2}\Big)^{-p\gamma}\leq 2^{p\gamma}.
            \end{align*}
  \end{proof}
\begin{proof}[Proof of Theorem \ref{exam}]
Since $(ii)$ above is independent of $N$, it follows that we also have the same estimate 
\begin{align}\label{estm}
M_{T^\gamma}=\sup_{t\in(0,1)}\big\|T^{\gamma}_t\big\| \leq 2^{\gamma}.
\end{align}

    This  easily implies that $(T_t^\gamma)_{t\geq 0}$ is a $C_0$-semigroup.

On the other hand, $(i)$ implies $\|T^\gamma_t\|\geq K_\gamma  t^{\gamma}$ for some $K_\gamma>0$. 

Thus it remains to check that $T^\gamma$ is Kreiss bounded that is equivalent to $T^\gamma$ being Cesàro bounded by Proposition \ref{kreisscescont}. As $T^\gamma$ is a direct sum, thanks to Lemma \ref{inv}, it follows directly from Lemma \ref{estAp} that 
\begin{align}\label{estK}
K_{T^\gamma}=\sup_{N\geq 2} K_{T^{\gamma,N}}\leq \frac C {1-\gamma}.
\end{align}
\end{proof}


\begin{rk}
The following quantitative estimate was obtained in \cite{Arnold}:
$$\|T_t\|\leq C M_T^2K_T^2 \frac t{(\ln t)^{1/p}}.$$
When $\gamma=1$, the semigroup $T^{1,N}$ shows that it is almost optimal. Indeed using \eqref{estm}, if $t=N$, we get
 $$M_{T^{1,N}}^2K_{T^{1,N}}^2\frac{N}{(\ln N)^{1/p}}\lesssim (\ln N)^{2-\frac 1p}N, \quad\textrm{but } \big\|T^{1,N}_N\big\| \gtrsim N.$$
\end{rk}

We end with a discrete example.
\begin{thm}\label{examd}
For $0<\gamma<1$, $T_1^\gamma$ is a positive Kreiss bounded operator on a commutative $L_p$-space with $\|{T^\gamma_1}^n\|\geq K_\gamma  n^{\gamma}$ for some $K_\gamma>0$ and $n\geq1$.
\end{thm}
\begin{proof} We keep the notation of the proof of Theorem \ref{exam}.
    We already know the norm estimates, it remains to prove that it is Kreiss. As $T_1^\gamma$ is a direct sum it suffices to prove it for all $T_1^{\gamma,N}$ uniformly in $N$. So fix $N\geq 2$ and consider the weight on $(0,2N)$ given by $\tilde \omega_\gamma^N(t)=\omega_\gamma^N([t])$. We have already seen that $1\leq \tilde \omega_\gamma^N(t)/ \omega_\gamma^N(t)\leq 2^{p\gamma}$. Thus loosing a factor $2^\gamma$, it suffices to evaluate the Kreiss constant of the translation by 1 on $L_p((0,2N),\tilde \omega_\gamma^N)$. We denote by $T_t$
    the translation by $t$ to lighten notation.  
    
    There is an obvious identification  $$L_p((0,2N),\tilde \omega_\gamma^N)=L_p\Big((0,1);\ell_p(\{0, \ldots, 2N-1\},\tilde \omega_\gamma^N)\Big)=L_p(0,1)\otimes_p \ell_p(\{0, \ldots, 2N-1\},\tilde \omega_\gamma^N).$$
    With it, $T_1$ acts like $Id\otimes  T_1^d$, where $T_1^d$ is a discrete translation by 1.  By the Fubini theorem, the Kreiss constant of $T_1$ is the same as that of $T_1^d$. Note also that 
    $\ell_p(\{0, \ldots, 2N-1\},\tilde \omega_\gamma^N)$ can be identified with piecewise constant functions on integer intervals in $L_p((0,2N),\tilde \omega_\gamma^N)$ and 
    $T_1^d$ is simply the restriction of $T_1$.
    
    By the continuous example and equivalence of the weights, there is a universal constant such that for all $T>0$ 
    $$\Big\| M_T=\frac 1 T \int_0^T  T_t dt\Big\|_{B\big(L_p((0,2N),\tilde \omega_\gamma^N)\big)}\leq K.$$ 
    
    The conditional expectation on piecewise constant functions on $L_p((0,2N),\tilde \omega_\gamma^N)$ is simply given by 
    $(\mathbb E f)(t)=\int_{[t]}^{[t+1]} f(u)du$ as the weight is constant on integer intervals. Let $e_i=\chi_{(i,i+1)}$for $i\in \{0,...,2N-1\}$. Then for $n\in N^*$, a direct computation gives that 
    $\mathbb E \Big(\int_0^1 T_u du\Big)(e_i)\geq \frac 12 e_{i-1}=\frac 12 T_1(e_i)$ (with $e_{-1}=0$). 
    
It follows that, for $k\in \bN^*$, $\mathbb E {M_k}_{|\ell_p(\{0, 2N-1\},\tilde \omega_\gamma^N)}\geq \frac 1 {2k} \sum_{l=1}^k (T^d_1)^l$. Hence we get that the Kreiss constant of $T_1^d$, that is the same as $T_1$, is bounded by at most $16K$ and the proof is over.
\end{proof}


\section*{Artificial Intelligence Statement }

\noindent No AI tools were employed to obtain the results of this paper.

\bibliographystyle{plain}
\bibliography{sample}

\end{document}